\documentclass[12pt]{amsart}
\usepackage{latexsym}
\usepackage{amssymb} 
\usepackage{mathrsfs}
\usepackage{amsmath}
\usepackage{latexsym}
\usepackage{delarray}
\usepackage{amssymb,amsmath,amsfonts,amsthm,
mathrsfs}

\numberwithin{equation}{section}
\theoremstyle{plain}
\newtheorem{thm}{Theorem}[section]
\newtheorem{lem}[thm]{Lemma}
\newtheorem{proposition}[thm]{Proposition}

\newtheorem{remark}[thm]{Remark}

\newcommand\Rn{{\mathbb R}^n}

\newcommand{\wt}[1]{\widetilde{#1}}

\newcommand{\Cinf}{\ensuremath{\mathcal{C}^\infty}}

\newcommand{\mb}[1]{\ensuremath{\mathbb{#1}}}
\newcommand{\N}{\mb{N}}

\newcommand{\R}{\mb{R}}
\newcommand{\C}{\mb{C}}

\newcommand{\lara}[1]{\langle #1 \rangle}

\newfont{\bigmath}{cmr12 at 13pt}

\newfont{\grecomath}{cmmi12 at 15pt}

\newcommand{\esp}{\mathrm{e}}

\newcommand{\beq}{\begin{equation}}
\newcommand{\eeq}{\end{equation}}

\newcommand{\eps}{\varepsilon}

\renewcommand{\Re}{\ensuremath{\mathrm{Re}}}

\title[On weakly hyperbolic equations with analytic principal part]{
On weakly hyperbolic equations with analytic principal part}

\author[Claudia Garetto]{Claudia Garetto}
\address{
  Claudia Garetto:
  \endgraf
  School of Mathematics
  \endgraf
  Loughborough University
  \endgraf
  Leicestershire, LE11 3TU
  \endgraf
  United Kingdom
  \endgraf
  {\it E-mail address} {\rm C.Garetto@lboro.ac.uk}
  }
\author[Michael Ruzhansky]{Michael Ruzhansky}
\address{
  Michael Ruzhansky:
  \endgraf
  Department of Mathematics
  \endgraf
  Imperial College London
  \endgraf
  180 Queen's Gate, London SW7 2AZ
  \endgraf
  United Kingdom
  \endgraf
  {\it E-mail address} {\rm m.ruzhansky@imperial.ac.uk}
  }

\thanks{The first author was supported by
the Imperial College Junior Research Fellowship.
The second author was supported by the
EPSRC Leadership Fellowship EP/G007233/1.
}
\date{}

\subjclass[2010]{Primary 35G10; 35L30; Secondary 46F05;}
\keywords{Hyperbolic equations, Gevrey spaces, ultradistributions.}

\begin{document}

\maketitle

\begin{abstract}
In this paper we show how to include low order terms in the $C^{\infty}$ well-posedness
results for weakly hyperbolic equations with analytic time-dependent coefficients.
This is achieved by doing a different reduction to a system from the previously used one.
We find the Levi conditions such that the $C^{\infty}$ well-posedness continues to hold.
\end{abstract}

\section{Introduction}

In this paper we study the Cauchy problem
\beq
\label{CP}
\displaystyle
\left\{
\begin{array}{cc}
M(t,D_t,D_x)u(t,x)=0,&\quad (t,x)\in(\delta,T+\delta)\times\R^n,\\
D^{j}_t u(t_0,x)=g_{j}(x),&\quad j=0,...,m-1,
\end{array}
\right.
\eeq
where $t_0\in(\delta, T+\delta)$, the equation
\[
M(t,D_t,D_x)u\equiv D^m_t u-\sum_{\substack{1\le j\le  m,\\ |\nu|\le j}} a_{\nu,j}(t)D^{m-j}_tD^\nu_x u=0
\]
is hyperbolic with $t$-dependent coefficients, analytic in the principal part, and
continuous in the lower order terms.

Equations of the form \eqref{CP} have been extensively studied in the literature.
If the equation \eqref{CP} is strictly hyperbolic and its coefficients are in the
H\"older class, $a_{\nu,j}\in C^{\alpha}$, $0<\alpha<1$, it was shown by the authors in
\cite[Remark 8]{GR:11} that the Cauchy problem \eqref{CP} is well-posed in
Gevrey classes $G^s(\Rn)$ provided that $1\leq s< 1+\frac{\alpha}{1-\alpha}$
(if $\alpha=1$, it is sufficient to assume the Lipschitz continuity of coefficients to get
the well-posedness in $G^s$ for all $s\geq 1$).
This extended to the general setting
the results for certain second order equations by Colombini, de Giorgi and Spagnolo
\cite{CDS} who have also shown
that the Gevrey index above is sharp. We also refer to \cite[Remark 16]{GR:11} for
the Gevrey-Beurling ultradistributional well-posedness for  $1\leq s\leq 1+\frac{\alpha}{1-\alpha}.$

Equations with higher regularity
$a_{\nu,j}\in C^{k}$ and lower order terms have been considered by
the authors in \cite{GR:12} in the
Gevrey classes, yielding also the well-posedness in $C^{\infty}$ in the case when
the coefficients $a_{\nu,j}$ of the principal part,
$|\nu|=j$, are analytic. There, the assumptions on \eqref{CP} have been formulated
in terms of the characteristic roots while the Levi conditions on lower order terms
can be expresses in terms of the coefficients of the operator $M$.
Also, in \cite{GR:12} the quasisymmetriser has been used while we use the
symmetriser in this paper.
We refer to \cite{GR:11} and \cite{GR:12} for a review of the existing literature
for this problem.

Recently, Jannelli and Taglialatela \cite{JT} treated the equation \eqref{CP} with analytic
coefficients, without lower order terms, proving the $C^{\infty}$ well-posedness under
assumptions that can be expressed entirely in terms of the coefficients of the operator
$M$. The purpose of this note is to show how to extend the result of
Jannelli and Taglialatela \cite{JT} to also include  lower order terms with Levi
conditions still formulated in terms of the coefficients of $M$. This will be
achieved in this paper by doing a different reduction of \eqref{CP} to the first
order system which will allow us to include the lower order terms in the energy.
Indeed, in the case of the homogeneous operators, the reduction used in
\cite{JT} was done to a system which is homogeneous of order one in $\xi.$
Such a reduction can not be used in the present context because the lower order
terms introduce singularities in the symbols at $\xi=0$. Instead, by employing
a reduction to a pseudo-differential system, also used in \cite{GR:11}, we are
able to avoid such singularities. Thus, the analysis in this note is based on a reduction
from \cite{GR:11}  combined with a number
of results from \cite{JT}, with a subsequent
treatment of lower order terms in the energy under Levi conditions introduced
below. An interesting feature is that the $C^{\infty}$ well-posedness holds for
analytic coefficients in the principal part and lower order terms which are only continuous.
We also give a result for bounded (and possibly discontinuous) lower order terms.

In \cite{CDS}, Colombini, de Giorgi and Spagnolo, and in \cite{CS}, Colombini
and Spagnolo gave examples of second order equations with time-dependent
coefficients which are not distributionally well-posed. In this paper,
we also prove the distributional well-posedness of \eqref{CP} in our setting.

In Section \ref{S:1} we introduce the notations and recall the results of \cite{JT}.
In Section \ref{S:2} we give our results. In Section \ref{S:3} we give the proofs,
and in Section \ref{S:4} we analyse the meaning of the assumptions on both
the principal part and lower order terms, and compare the obtained results with
those in \cite{GR:12}.

\section{Preliminaries}
\label{S:1}

We begin by recalling the theorem proved in \cite{JT} for the Cauchy problem
\beq
\label{CPJT}
\left\{
\begin{array}{cc}
L(t,\partial_t,\partial_x)u(t,x)=0,\quad \text{for $(t,x)\in(\delta,T+\delta)\times\R^n$},\\
\partial^j_t u(t_0,x)=g_j(x),\quad \text{for $x\in\R^n$, $j=0,...,m-1$},
\end{array}
\right.
\eeq
where $t_0\in(\delta,T+\delta)$, and
\[
L(t,\partial_t,\partial_x)u(t,x)=\partial^m_t u(t,x)-\sum_{1\le j\le  m, |\nu|=j}
a_{\nu,j}(t)\partial^{m-j}_t\partial^\nu_x u(t,x)
\]
is homogeneous of order $m$.
This requires some preliminary notions which are collected in the sequel.
Let
\[
P(t,\tau,\xi)=\tau^m-\sum_{1\le j\le  m, |\nu|=j} a_{\nu,j}(t)\xi^\nu\tau^{m-j}=
\tau^m-\sum_{j=1}^m h_j(t,\xi)|\xi|^j\tau^{m-j}
\]
be the symbol of $i^{-m} L$. This is also the principal symbol of $M$ in \eqref{CP}. Let
\beq
\label{AJT}
A_0(t,\xi)=\left(
    \begin{array}{ccccc}
      0 & 1 & 0 & \dots & 0\\
      0 & 0 & 1 & \dots & 0 \\
      \dots & \dots & \dots & \dots & 1 \\
      h_m & h_{m-1} & \dots & \dots & h_1 \\
    \end{array}
  \right)
\eeq
be the companion matrix of $P(t,\tau,\xi/|\xi|)$. By construction the matrix $A_0(t,\xi)$ is
homogeneous of oder zero in $\xi$,
 and the eigenvalues of $A_0(t,\xi)|\xi|$ are the characteristic roots
 $\tau_1(t,\xi)$,...,$\tau_m(t,\xi)$ of $P(t,\tau,\xi)$.

For the moment let us fix $t\in(\delta, T+\delta)$ and $\xi$ so that $P$ is a
polynomial in $\tau$ with constant coefficients. In \cite{J:89} Jannelli
constructed a real symmetric $m\times m$ matrix $Q$ which is weakly positive definite
if and only if $P$ is weakly hyperbolic. This $Q$ is called the
\emph{standard symmetriser} of $P$. Note that the entries of $Q$
are fixed polynomials functions of $h_1,..., h_m$ such that
\beq
\label{QA}
QA_0-A_0^\ast Q=0
\eeq
and
\[
\det Q =\prod_{1\le k<j\le m}(\tau_j-\tau_k)^2.
\]
Let $Q_j$ be the principal $j\times j$ minor of $Q$ obtained by removing the first $m-j$ rows and columns of $Q$ and let $\Delta_j$ its determinant. When $j=m$ we use the notations $Q$ and $\Delta$ instead of $Q_m$ and $\Delta_m$. The hyperbolicity of $P$ can be seen at the level of the symmetriser $Q$ and of its minors as stated in the following proposition (see \cite{J:09}).
\begin{proposition}
\label{propJan}
\leavevmode
\begin{itemize}
\item[(i)] $P$ is strictly hyperbolic if and only if $\Delta_j>0$ for all $j=1,...,m$.
\item[(ii)] $P$ is weakly hyperbolic if and only if there exists $r<m$ such that
    \[
    \Delta=\Delta_{m-1}=...=\Delta_{r+1}=0
    \]
    and $\Delta_r>0,...,\Delta_1>0$. (In this case there are exactly $r$ distinct roots).
\end{itemize}
\end{proposition}
Clearly, when $t$ and $\xi$ vary in $(\delta, T+\delta)$ and $\R^n$, respectively, $\Delta_r$ becomes a function $\Delta_r(t,\xi)$ homogeneous of degree $0$ in $\xi$ and analytic in $t$. When $\Delta_r$ is not identically zero one can define the function
\[
\widetilde{\Delta}_r(t,\xi)=\Delta_r(t,\xi)+\frac{(\partial_t\Delta_r(t,\xi))^2}{\Delta_r(t,\xi)},
\]
which is homogeneous of degree $0$ in $\xi$ as well, analytic on the interval $(\delta, T+\delta)$. In addition, the following property holds for $\Delta$ and $\wt{\Delta}$: if $t\mapsto\Delta(t,\xi)$ vanishes of order $2k$ at
a point $t'$ then $t\mapsto\widetilde{\Delta}(t,\xi)$ vanishes of order $2k-2$ at $t'$.

Note that estimating the quotient $\lara{\partial_t QV,V}/\lara{QV,V}$ is equivalent to estimating
the roots of the generalised Hamilton-Cayley polynomial
\beq
\label{HamCay}
\det(\lambda Q-\partial_t Q)=\sum_{j=0}^m d_j(t)\lambda^{m-j}
\eeq
of $Q$ and $\partial_t Q$, where $d_0=\det Q$, $d_1=-\partial_t(\det Q)$, $d_m=(-1)^m \det (\partial_t Q)$ and, if $m\ge 2$, $d_2=\frac{1}{2}{\rm trace}(\partial_t Q\partial_t (Q^{{\rm co}}))$, where $Q^{{\rm co}}$ is the cofactor matrix of $Q$. We recall that the cofactor of $Q$ is the matrix with entries
$q^{{\rm co}}_{ij}=(-1)^{i+j}d_{ij}$, where $d_{ij}$ is the determinant of the submatrix obtained from $Q$ by removing the $i$-th row and the $j$-th column. Finally, from the known identity
\[
\lambda_1^2+\cdots +\lambda_m^2=\biggl(\frac{d_1}{d_0}\biggr)^2-2\frac{d_2}{d_0},
\]
valid for the roots $\lambda_j$, $j=1,\dots,m$, of the generalised Hamilton-Cayley polynomial, we see that $d_2$ plays a fundamental role when one wants to estimate $\lara{\partial_t QV,V}/\lara{QV,V}$. We call
\[
\psi(t,\xi):=d_2(t,\xi)
\]
defined as above the \emph{check function} of $Q$. Replacing $Q$ with $Q_j$ in the
definition of $d_2$ we define the check function $\psi_j(t,\xi)$ of $Q_j$.
Clearly, $\psi_j(t,\xi)$ is homogeneous of order zero in $\xi$.
Note that when $m=1$ the check function $\psi$ is set to be identically $0$.
We are now ready to state the $\Cinf$ well-posedness theorem of
Jannelli and Taglialatela given in \cite{JT}.

Since the purpose of this note is to describe the possibility of adding lower order terms to
$L$ we will avoid long technicalities and will focus on
the non-degenerate case, i.e.,  the case with $\Delta(\cdot,\xi)\not\equiv 0$ is
not identically zero in $(\delta, T+\delta)$.

We skip the treatment of the Cauchy problem \eqref{CPJT} in the degenerate case since it is
lengthy but remark that the analysis can be carried out in this case as well along the lines
of the analysis of the general case in \cite{GR:12}.
In the present context, it would make use of $\Delta_r(\cdot,\xi)$ and the corresponding
check function $\psi_r(t,\xi)$, where $r=r(\xi)$ is the greatest integer such that
$\Delta_r(\cdot,\xi)\not\equiv 0$ in $(\delta, T+\delta)$.
For more details on these for the case of homogeneous $L$,
see Theorem 1 and Section 3 in \cite{JT}.

\begin{thm}[\cite{JT}]
\label{theoJT}
Let $L(t,\partial_t,\partial_x)$ as in \eqref{CPJT} be a weakly hyperbolic homogeneous operator with analytic coefficients in $(\delta,T+\delta)$. Let $P(t,\tau,\xi)$ be the characteristic polynomial and
$A_{0}(t,\xi)$ the companion matrix of $P(t,\tau,\xi/|\xi|)$. Let $Q(t,\xi)$ be its symmetriser and $\psi(t,\xi)$ the check function of $Q(t,\xi)$. Let $[a,b]\subset(\delta, T+\delta)$.
Assume that there exists a constant $C>0$ independent of $\xi$ such that
\beq
\label{JT1}
|\psi(t,\xi)|\le C\widetilde{\Delta}(t,\xi)
\eeq
holds for all $t\in[a,b]$. Then, the Cauchy problem \eqref{CPJT} is $\Cinf$ well-posed in $[a,b]$.
\end{thm}
We can write condition \eqref{JT1} in a different way by introducing the set $\Sigma(\xi)=\{t_1,...,t_{N(\xi)}:\, \Delta(t_j,\xi)=0\}$ and the function
\[
\mathcal{Z}(t,\xi)=\begin{cases} \prod_{j=1}^{N(\xi)}|t-t_j| & \text{if $\Sigma(\xi)=\{t_1,...,t_{N(\xi)}\}$},\\
1 & \text{if $\Sigma(\xi)=\emptyset$}.
\end{cases}
\]
Note that by the analyticity of $\Delta$ in $t$ it follows that the function $N(\xi)$ is locally bounded (see the proof of Lemma 6.2 in \cite{GR:12}). Using again the fact that $\Delta(t,\xi)$ is analytic in $t$ and homogeneous of order $0$ in $\xi$ one can prove that there exist constants $c_1>0$ and $c_2>0$ independent of $\xi$ such that
\[
c_1\frac{\Delta(t,\xi)}{\mathcal{Z}^2(t,\xi)}\le \widetilde{\Delta}(t,\xi)\le c_2\frac{\Delta(t,\xi)}{\mathcal{Z}^2(t,\xi)},
\]
for all $t\in(\delta,T+\delta)$ and $\xi\neq 0$. Hence, \eqref{JT1} can be reformulated as follows: there exists a constant $C>0$ independent of $\xi$ such that
\beq
\label{JT2}
\mathcal{Z}^2(t,\xi)|\psi(t,\xi)|\le C\Delta(t,\xi),
\eeq
for all $t\in[a,b]$ and $\xi\neq 0$. This extends to any space dimension the one-dimensional observation of Jannelli and Tagliatela made in \cite[p. 1000]{JT}.

\section{Results}
\label{S:2}

We are now ready to study the Cauchy problem \eqref{CP} or, in other words,
to add lower order terms to the equation in \eqref{CPJT}.
First, we describe the reduction of \eqref{CP} to
a system since we will be making use of the symmetriser of the corresponding companion
matrix. We rewrite the equation
\[
M(t,D_t,D_x)u\equiv D^m_t u-\sum_{\substack{1\le j\le  m,\\ |\nu|\le j}} a_{\nu,j}(t)D^{m-j}_t D^\nu_x u=0
\]
as
\[
D^m_t u(t,x)=\sum_{\substack{1\le j\le  m,\\ |\nu|=j}} a_{\nu,j}(t)D^{m-j}_t D^\nu_x u(t,x)+\sum_{\substack{1\le j\le  m,\\ |\nu|\le j-1}}a_{\nu,j}(t)D^{m-j}_t D^\nu_x u(t,x).
\]
First of all we perform the standard reduction to a system of pseudo-differential equations as in \cite{GR:11} by setting
\[
u_l=D_t^{l-1}\lara{D_x}^{m-l}u,
\]
with $l=1,...,m$, where $\lara{D_x}$ is the pseudo-differential operator with symbol
$\lara{\xi}=(1+|\xi|^2)^{\frac{1}{2}}$. This transformation makes the $m$th-order
equation above equivalent to the first order system
\begin{equation}\label{EQ:sys}
D_tU = A_1(t,D_x)U+B(t,D_x)U,
\end{equation}
where $U$ is the column vector with entries $u_l$, $A_1(t,D_x)$ has symbol matrix
\[
A_1(t,\xi)=\lara{\xi}A(t,\xi)=\lara{\xi}\left(
    \begin{array}{ccccc}
      0 & 1 & 0 & \dots & 0\\
      0 & 0 & 1 & \dots & 0 \\
      \dots & \dots & \dots & \dots & 1 \\
      a_{1} & a_{2} & \dots & \dots & a_m \\
    \end{array}
  \right),
\]
where
\[
a_j(t,\xi)=\sum_{|\nu|=m-j+1}a_{\nu,m-j+1}(t)\xi^\nu\lara{\xi}^{j-m-1},
\]
and $B(t,D_x)$ has symbol matrix
\[
B(t,\xi)=\left(
    \begin{array}{ccccc}
      0 & 0 & 0 & \dots & 0\\
      0 & 0 & 0& \dots & 0 \\
      \dots & \dots & \dots & \dots & 0 \\
      b_1 & b_2 & \dots & \dots & b_m \\
    \end{array}
  \right),
\]
with
\beq
\label{bi}
b_j(t,\xi)=\sum_{|\nu|\le m-j}a_{\nu,m-j+1}(t)\xi^\nu\lara{\xi}^{j-m}.
\eeq
The initial conditions $D^j_t u(t_0,x)=g_j(x)$, $j=0,...,m-1$, are transformed into
\[
u_l(t_0,x)=\lara{D_x}^{m-l}g_{l-1}(x),
\]
for $l=1,...,m$, and generate the column vector $U_0(x)$.
In the following theorem we use functions $\psi$ and $\widetilde{\Delta}$ which
have been defined at the beginning of Section \ref{S:1}.

\begin{thm}
\label{theoGRm}
Let $M(t,D_t,D_x)$ as in \eqref{CP} be a weakly hyperbolic operator with analytic
coefficients in $(\delta,T+\delta)$ in the principal part, and lower order terms
continuous in $t$. Let $Q(t,\xi)=\{q_{ij}(t,\xi)\}_{i,j=1}^{m}$
be the symmetriser of the matrix $A(t,\xi)$ of the principal part,
$\Delta$ its determinant and $\psi(t,\xi)$ its check function. Let
$\Delta(\cdot,\xi)\not\equiv 0$ in $(\delta, T+\delta)$ for all $\xi$ with $|\xi|\ge 1$ and let $[a,b]\subset(\delta, T+\delta)$.
Assume that
\begin{itemize}
\item[(i)]
there exists a constant $C_1>0$ such that
\beq
\label{GR1m}
|\psi(t,\xi)|\le C_1\widetilde{\Delta}(t,\xi)
\eeq
holds for all $t\in[a,b]$ and $|\xi|\ge 1$;
\item[(ii)] the Levi condition
\begin{equation}\label{EQ:Leco}
|(q_{im}b_j-\overline{b_i}q_{jm})(t,\xi)|\le c\,\Delta(t,\xi)
\end{equation}
holds for all $1\le i,j\le m$, $t\in[a,b]$ and $|\xi|\ge 1$.
\end{itemize}
Then the Cauchy problem \eqref{CP} is $\Cinf$ well-posed in $[a,b]$ with initial data at $t_0=a$, and it is also
well-posed in $\mathcal D'(\Rn).$
\end{thm}
One of the features of this result is that we can allow lower order terms to be complex-valued.
In Remark \ref{rem_B} we will comment on a small simplification of the Levi conditions (ii)
if the matrix $B$ is real.
In the case when lower order terms are discontinuous but still bounded, we have the following
counterpart of the result above. The proof is similar to that of Theorem 1.4 in \cite{GR:12}.

\begin{thm}
\label{theoGRm1}
Assume the conditions of Theorem \ref{theoGRm}, but instead of assuming that lower order
terms are continuous, assume only that they are bounded, i.e
that
$a_{\nu,j}\in L^{\infty}((\delta,T+\delta))$, for all $|\nu|\leq j-1$ with $1\le j\le m$. Then the statement
remains true for smooth Cauchy data provided that we replace the well-posedness conclusion
$u\in C^m([a,b];C^{\infty}(\R^n))$ by
\[
u\in C^{m-1}([a,b];C^{\infty}(\R^n))\cap W^{\infty,m}([a,b];C^{\infty}(\R^n)),
\]
where $W^{\infty,m}$ is the Sobolev space with $m$ derivatives in $L^{\infty}.$
\end{thm}
The same distributional conclusion as we had in Theorem \ref{theoGRm} also holds in
Theorem \ref{theoGRm1}, with the solution
\[
u\in C^{m-1}([a,b];\mathcal D'(\R^n))\cap W^{\infty,m}([a,b];\mathcal D'(\R^n)),
\]
provided that the Cauchy data are all in $\mathcal D'(\R^n).$

\section{Proofs}
\label{S:3}

Note that the eigenvalues of the matrix $A_1=\lara{\xi}A(t,\xi)$ are the roots of the characteristic polynomial $P(t,\tau,\xi)$ and that the entries of the matrix $A$ are related to the entries of the matrix $A_0$ in \eqref{AJT} by the formula  $a_j(t,\xi)\lara{\xi}^{m-j+1}=h_{m-j+1}(t,\xi)|\xi|^{m-j+1}$.
Applying the Fourier transform to the system \eqref{EQ:sys} we obtain the system
\beq
\label{eqV}
\begin{split}
D_t V(t,\xi)&=\lara{\xi}A(t,\xi)V(t,\xi)+B(t,\xi)V(t,\xi),\\
V(t_0,\xi)&= V_0(\xi),
\end{split}
\eeq
where $V=\mathcal{F}_{x\to\xi}U$ and $V_0=\mathcal{F}_{x\to\xi} U_0$.

Note that by performing a reduction to a system of pseudo-differential equations the symmetriser $Q(t,\xi)$, defined as in \cite[Section 3]{GR:12}, is a matrix of $0$-order symbols which can be expressed in terms of the rescaled roots $\tau_j(t,\xi){\lara{\xi}}^{-1}$ (or eigenvalues of the matrix $A$).
More precisely, the entries $q_{ij}$ of the symmetriser $Q(t,\xi)$ are polynomials in
$\tau_1{\lara{\xi}}^{-1},...,\tau_m{\lara{\xi}}^{-1}$.
Making use of the concept of the Bezout matrix associated to
the couple of polynomials $(P,\partial_\tau P)$ it is also possible to express the entries
of the symmetriser in terms of the coefficients $h_j$, $j=1,...,m$.
For further details we refer the reader to \cite[p. 998]{JT}.
We begin by proving some basic properties of the symmetriser which will be employed to prove
the $\Cinf$ well-posedness.

\subsection{The symmetriser $Q$}
It is useful to make a comparison between the symmetriser $Q$ of the matrix $A$ and the symmetriser $Q_0$ with homogeneous entries of order $0$ employed by Jannelli and Taglialatela in \cite{JT}. These two matrices are both symmetric with polynomial entries in $\tau_1\lara{\xi}^{-1}$,...,$\tau_m\lara{\xi}^{-1}$ and $\tau_1|\xi|^{-1}$,...,$\tau_m|\xi|^{-1}$, respectively, as defined in \cite[Section 3]{GR:12}. By construction, $QA-A^\ast Q=0$ and
\begin{multline}
\label{detq}
\Delta(t,\xi)=\prod_{1\le i<j\le m}(\tau_j(t,\xi)\lara{\xi}^{-1}-\tau_i(t,\xi)\lara{\xi}^{-1})\\
=\lara{\xi}^{-m(m-1)}\prod_{1\le i<j\le m}(\tau_j(t,\xi)-\tau_i(t,\xi))= \lara{\xi}^{-m(m-1)}|\xi|^{m(m-1)}\Delta_0(t,\xi),
\end{multline}
where $\Delta$ and $\Delta_0$ are the determinants of $Q$ and $Q_0$, respectively, and $\Delta_0$ is expressed in terms of the $0$-homogeneous roots $\tau_i/|\xi|$. The following lemma on symmetric positive semi-definite matrices will be in the sequel applied to $Q$.
\begin{lem}
\label{lemmaJT}
Let $N(t)$ be any symmetric positive semi-definite matrix with bounded coefficients on an interval $[a,b]$. Then, there exist two positive constants $c_1$ and $c_2$ depending only on the $L^\infty$-norm of the entries of $N(t)$ such that
\[
c_1\det N(t)|V|^2\le \lara{N(t)V,V}\le c_2|V|^2
\]
holds for all $t\in[a,b]$ and $V\in\C^m$.
\end{lem}
Since $Q(t,\xi)$ is a positive semi-definite matrix with eigenvalues $\lambda_i(t,\xi)$,
which satisfy symbol estimates of order $0$ in $\xi$ and we can assume them ordered,
i.e., $\lambda_1\le\lambda_2\le...\le\lambda_m$, there exists a constant $c_0>0$ such that
\[
0\le\lambda_1(t,\xi)\le\lambda_2(t,\xi)\le\cdots\le\lambda_m(t,\xi)\le c_0
\]
holds for all $t\in[a,b]$ and $\xi\in\R^n$. It follows that when $\det Q(t,\xi)>0$ we can write
\[
\lara{QV,V}\ge \lambda_1|V|^2=\frac{\det Q(t,\xi)}{\lambda_2(t,\xi)\cdots\lambda_m(t,\xi)}|V|^2\ge \frac{\det Q(t,\xi)}{c_0^{m-1}}|V|^2.
\]
It follows that Lemma \ref{lemmaJT} holds for the matrix $Q(t,\xi)$. More precisely,
\begin{lem}
\label{lemmaGR}
Let $Q(t,\xi)$ be the symmetriser of the weakly hyperbolic matrix $A(t,\xi)$ defined above. Then, there exist two positive constants $c_1$ and $c_2$ such that
\[
c_1\det Q(t,\xi)|V|^2\le \lara{Q(t,\xi)V,V}\le c_2|V|^2
\]
holds for all $t\in[a,b]$, $\xi\in\R^n$ and $V\in\C^m$.
\end{lem}
Let $I$ be a closed interval in $\R$. We recall (see also \cite[p.~1003-1004]{JT})
that if $B(t)$ and $C(t)$ are two real symmetric $m\times m$ matrices, $C(t)$ is nonnegative and $\det C(t)$ has only isolated zeros then
\beq
\label{BC}
\frac{\lara{B(t)V,V}}{\lara{C(t)V,V}}\in L^\infty(I\times (\C^m\setminus 0))
\eeq
if and only if the roots $\lambda_i$ of the generalised Hamilton-Cayley polynomial
\[
\det(\lambda C(t)-B(t))=\sum_{h=0}^m d_h(t)\lambda^{m-h}
\]
of $B(t)$ and $C(t)$, are bounded functions of $t$. Since
\[
\sum_{i=1}^m\lambda_i^2(t)=\frac{d_1^2(t)}{d_0^2(t)}-2\frac{d_2(t)}{d_0(t)},
\]
we conclude that \eqref{BC} holds if and only if
\begin{equation}\label{EQ:ds}
\frac{d_1^2(t)}{d_0^2(t)}-2\frac{d_2(t)}{d_0(t)}
\end{equation}
is bounded. We are now ready to prove the following lemma.
\begin{lem}
\label{lemmaEGR}
Let $Q(t,\xi)$ be the symmetriser of the matrix $A(t,\xi)$. Let $\Delta(t,\xi)=\det Q(t,\xi)$, $\wt{\Delta}(t,\xi)=\Delta(t,\xi)+(\partial_t \Delta(t,\xi))^2/\Delta(t,\xi)$, $\psi(t,\xi)$ the check function of $Q(t,\xi)$. Let $I$ be a closed interval of $\R$.
Then,
\beq
\label{E1}
\sqrt{\frac{\Delta(t,\xi)}{\wt{\Delta}(t,\xi)}}\frac{\lara{\partial_t Q(t,\xi)V,V}}{\lara{Q(t,\xi)V,V}}\in L^\infty (I\times \R^n\times\C^m\setminus  0)
\eeq
if and only if
\beq
\label{C1}
\frac{\psi(t,\xi)}{\wt{\Delta}(t,\xi)}\in L^\infty (I\times \R^n).
\eeq
\end{lem}
\begin{proof}
Since
\[
\det \biggl(\lambda Q-\partial_t Q\sqrt{\frac{\Delta}{\wt{\Delta}}}\biggr)=\sqrt{\frac{\Delta}{\wt{\Delta}}}^{\,m}\det\biggl
(\lambda\sqrt{\frac{\wt{\Delta}}{\Delta}} Q-\partial_t Q\biggr),
\]
by \eqref{HamCay} and the definition of $d_0$, $d_1$ and $d_2$ we get
\[
\det \biggl(\lambda Q-\partial_t Q\sqrt{\frac{\Delta}{\wt{\Delta}}}\biggr)=\Delta\lambda^m-\partial_t\Delta\sqrt{\frac{\Delta}{\wt{\Delta}}}\lambda^{m-1}
+\psi\frac{\Delta}{\wt{\Delta}}\lambda^{m-2}+\sum_{h=3}^m d_h\biggl(\frac{\Delta}{\wt{\Delta}}\biggr)^{\frac{h}{2}}
\lambda^{m-h}.
\]
Applying \eqref{BC} to $B=\partial_t Q\sqrt{\frac{\Delta}{\wt{\Delta}}}$ and $C=Q$ we obtain that \eqref{E1} holds if and only if
\[
\frac{\partial_t\Delta\sqrt{\frac{\Delta}{\wt{\Delta}}}}{\Delta}
\]
and
\[
\frac{\psi\frac{\Delta}{\wt{\Delta}}}{\Delta}
\]
are bounded functions on $I\times \R^n$ and there is no cancellation of unbounded terms in
\eqref{EQ:ds}. Since
\[
\biggl|\frac{\partial_t\Delta\sqrt{\frac{\Delta}{\wt{\Delta}}}}{\Delta}\biggr|\le 1
\]
we have that \eqref{E1} is equivalent to \eqref{C1}.
\end{proof}

\subsection{A technical lemma about real analytic functions on a real interval}
Recalling the relationship between $\Delta$ and $\Delta_0$ in \eqref{detq} we can apply Proposition 4.1 in \cite{JT} to the determinant $\Delta(t,\xi)$ of the symmetriser $Q(t,\xi)$. This yields the following statement.
\begin{proposition}
\label{propJT}
 Let $\Delta(t,\xi)$ be as above. Suppose that $\Delta(t,\xi)\not\equiv 0$. Then,
 \begin{itemize}
 \item[(i)] there exists $X\subset\mathbb{S}^{n-1}$ such that $\Delta(t,\xi)\not\equiv 0$ in $(\delta, T+\delta)$ for any ${\xi}\in X$ and the set $\mathbb{S}^{n-1}\setminus X$ is negligible with respect to the Hausdorff $(n-1)$-measure;
 \item[(ii)] for any $[a,b]\subset(\delta, T+\delta)$ there exist $c_1,c_2>0$ and $p,q\in\N$ such that for any $\xi\in X$ and any $\eps\in(0,\esp^{-1}]$ there exists $A_{\xi,\eps}\subset[a,b]$ such that:
\begin{itemize}
\item[-] $A_{\xi,\eps}$ is a union of at most $p$ disjoint intervals,
\item[-] $m(A_{\xi,\eps})\le\eps$,
\item[-] $\min_{t\in[a,b]\setminus A_{\xi,\eps}} \Delta(t,\xi)\ge c_1\eps^{2q}\Vert\Delta(\cdot,\xi)\Vert_{L^\infty([a,b])}$,
\item[-]
\[
\int_{t\in[a,b]\setminus A_{\xi,\eps}}\frac{|\partial_t\Delta(t,\xi)|}{\Delta(t,\xi)}\, dt\le c_2\log\frac{1}{\eps}.
\]
\end{itemize}

\end{itemize}
\end{proposition}
Proposition \ref{propJT} provides a partition of the interval $[a,b]$ which is crucial in the proof of the $\Cinf$ well-posedness in the next subsection.

\subsection{$\Cinf$ well-posedness of the Cauchy problem \eqref{CP}}
To prove the $\Cinf$ well-posedness of the Cauchy problem \eqref{CP} in the form \eqref{eqV} we make use of the energy
\[
E(t,\xi)=\begin{cases}
|V(t,\xi)|^2 & \text{for $t\in A_{\xi/|\xi|,\eps}$ and $\xi/|\xi|\in X$},\\
\lara{Q(t,\xi)V(t,\xi),V(t,\xi)} & \text{for $t\in[a,b]\setminus A_{\xi/|\xi|,\eps}$ and $\xi/|\xi|\in X$,}
\end{cases}
\]
defined for $t\in[a,b]$, $\xi\in\R^n$ with $\xi/|\xi|\in X$, and $\eps\in(0,\esp^{-1}]$.
Note that $\Delta(t,\xi)>0$ when $t\in[a,b]\setminus A_{\xi/|\xi|,\eps}$ and $\xi/|\xi|\in X$,
and $[a,b]\setminus A_{\xi/|\xi|,\eps}$ is a finite union of at most $p$ closed intervals $[c_i,d_i]$.
The set $A_{\xi/|\xi|,\eps}$ is a finite union of open intervals whose total length does not
exceed $\eps$. We now distinguish between Kovalevskian energy and hyperbolic energy.

\subsubsection{The Kovalevskian energy}
Let $t\in[t',t^{''}]\subseteq A_{\xi/|\xi|,\eps}$ and $\xi/|\xi|\in X$. Hence
\[
\begin{split}
\partial_t E(t,\xi)&= 2\Re\lara{V(t,\xi),\partial_t V(t,\xi)}\\
& = 2\Re\lara{V(t,\xi),i\lara{\xi}A(t,\xi)V(t,\xi)+iB(t,\xi)V(t,\xi)}\le 2(c_A\lara{\xi}+c_B)E(t,\xi).
\end{split}
\]
By Gronwall's Lemma on $[t',t^{''}]$ we get
\beq
\label{estEk}
|V(t,\xi)|\le \esp^{(c_A\lara{\xi}+c_B)(t-t')}|V(t',\xi)|\le c\,\esp^{c\lara{\xi}(t-t')}|V(t',\xi)|.
\eeq

\subsubsection{The hyperbolic energy}
Let us work on any subinterval $[c_i,d_i]$ of $[a,b]\setminus A_{\xi/|\xi|,\eps}$. Assuming $\xi/|\xi|\in X$,
we have that $\Delta(t,\xi)>0$ on $[c_i,d_i]$. Under the hypothesis \eqref{JT1}
using \eqref{QA}, Lemma \ref{lemmaEGR} and Lemma \ref{lemmaGR}, we have that
\[
\begin{split}
&\partial_t E(t,\xi)=\lara{\partial_t Q(t,\xi)V(t,\xi),V(t,\xi)}\\
&+\lara{Q(t,\xi)\partial_tV(t,\xi),V(t,\xi)}+\lara{Q(t,\xi)V(t,\xi),\partial_t V(t,\xi)}\\
&=\frac{\lara{\partial_t Q(t,\xi)V(t,\xi),V(t,\xi)}}{\lara{Q(t,\xi)V,V}}E(t,\xi)+\lara{Q(t,\xi)(i\lara{\xi}A(t,\xi)+iB(t,\xi))V(t,\xi),V(t,\xi)}\\
&+\lara{Q(t,\xi)V(t,\xi),(i\lara{\xi}A(t,\xi)+iB(t,\xi))V(t,\xi)}\\
&=\frac{\lara{\partial_t Q(t,\xi)V(t,\xi),V(t,\xi)}}{\lara{Q(t,\xi)V,V}}E(t,\xi)+i\lara{(Q(t,\xi)B(t,\xi)-B^\ast(t,\xi)Q(t,\xi))V(t,\xi),V(t,\xi)}\\
&\le C\sqrt{\frac{\wt{\Delta}(t,\xi)}{\Delta(t,\xi)}}E(t,\xi)+
|\lara{(Q(t,\xi)B(t,\xi)-B^\ast(t,\xi)Q(t,\xi))V(t,\xi),V(t,\xi)}|.\\
\end{split}
\]
At this point it is clear that if we have Levi conditions on the matrix $B$ such that
\beq
\label{LCB}
\frac{|\lara{(Q(t,\xi)B(t,\xi)-B^\ast(t,\xi)Q(t,\xi))V(t,\xi),V(t,\xi)}|}{\lara{Q(t,\xi)V,V}}\in L^\infty([a,b]\times \R^n\times\C^m\setminus 0),
\eeq
then there exists a constant $c>0$ such that
\beq
\label{Elot}
\partial_t E(t,\xi)\le c\biggl(1+\frac{|\partial_t\Delta(t,\xi)|}{\Delta(t,\xi)}\biggr)E(t,\xi)
\eeq
holds for $t\in[c_i,d_i]$ and $\xi\in\R^n$ with $\xi/|\xi|\in X$.
Before proving that this energy estimate yields the $\Cinf$ well-posedness
of the Cauchy problem \eqref{CP} let us show that the Levi conditions
\eqref{EQ:Leco} for the matrix $B$ guarantee \eqref{LCB}.
\begin{proposition}
\label{propLC}
Let $b_i$ and $q_{ij}$, $i,j=1,...,m$, be the entries of the matrix $B$ and $Q$ defined above. If there exists a constant $c>0$ such that
\beq
\label{LCb}
|(q_{im}b_j-\overline{b_i}q_{jm})(t,\xi)|\le c\Delta(t,\xi)
\eeq
holds for all $1\le i,j\le m$, $t\in[a,b]$ and $\xi\in\R^n$, then \eqref{LCB} holds.
\end{proposition}
\begin{proof}
We begin by observing that the matrix $QB-B^\ast Q$ has entries $d_{ij}=q_{im}b_j-\overline{b_i}q_{jm}$. It follows that 
\[
\lara{(QB-B^\ast Q)V,V}=\sum_{i=1}^m \sum_{j=1}^m(q_{im}b_j-\overline{b_i}q_{jm})V_j\overline{V_i}
\]
Hence,
\[
|\lara{(QB-B^\ast Q)V,V}|\le \sum_{i,j=1,...,m}|q_{im}b_j-\overline{b_i}q_{jm}|(|V_i|^2+|V_j|^2)
\]
and by the hypothesis \eqref{LCb} we get
\[
|\lara{(QB-B^\ast Q)V,V}(t,\xi)|\le c \Delta(t,\xi) |V(t,\xi)|^2,
\]
for some constant $c>0$, uniformly in $t\in[a,b]$ and $\xi\in\R^n$. Finally, combining this estimate with the bound from below of Lemma \ref{lemmaGR} we obtain \eqref{LCB}.
\end{proof}
\begin{remark}
\label{rem_B}
Note that when the lower order terms are real-valued then the matrix $QB-B^\ast Q$ is skew-symmetric. This means that
$d_{ij}=q_{im}b_j-\overline{b_i}q_{jm}$ is identically zero when $i=j$ and $d_{ij}=-d_{ji}$. It follows that the Levi conditions \eqref{LCB} can be rewritten as
\beq
\label{LCbr}
|(q_{im}b_j-b_iq_{jm})(t,\xi)|\le c\Delta(t,\xi),
\eeq
for $1\le i<j\le m$, $t\in[a,b]$ and $\xi\in\R^n$.
\end{remark}

We are now ready to prove Theorem \ref{theoGRm}.

\begin{proof}[Proof of Theorem \ref{theoGRm}]
First of all, by the finite speed of propagation for hyperbolic equations we can always assume that the Cauchy data in \eqref{CP} are compactly supported. We refer to the Kovalevskian energy and the hyperbolic energy introduced above. We note that in the energies in consideration we can assume $|\xi|\ge 1$ since the continuity of $V(t,\xi)$ in $\xi$ implies that both energies are bounded for $|\xi|\le 1$. In particular, the Levi condition \eqref{LCb} for $|\xi|\ge 1$ yields the energy estimate \eqref{Elot} for $|\xi|\ge 1$. Hence, by Gronwall's Lemma on $[c_i,d_i]$ we get the inequality
\beq
\label{estEhyp}
E(t,\xi)\le \esp^{c(d_i-c_i)}\exp\biggl(c\int_{c_i}^t\frac{|\partial_s\Delta(s,\xi)|}{\Delta(s,\xi)}\,ds\biggr)E(c_i,\xi).
\eeq
Note that by Proposition \ref{propJT}, (ii), we have
\[
\Delta(t,\xi)\ge \min_{s\in [a,b]\setminus A_{\xi,\eps}}\Delta(s,\xi)\ge c_1\eps^{2q}\Vert \Delta(\cdot,\xi)\Vert_{L^\infty([a,b])},
\]
for all $t\in[c_i,d_i]$. Hence, applying Lemma \ref{lemmaGR} to \eqref{estEhyp}
we have that there exists a constant $C>0$ such that
\begin{multline}
\label{estEhyp2}
|V(t,\xi)|^2\le C\frac{1}{\eps^{2q}\Vert \Delta(\cdot,\xi)\Vert_{L^\infty([a,b])}}\exp\biggl(\int_{c_i}^t
\frac{|\partial_s\Delta(s,\xi)|}{\Delta(s,\xi)}\,ds\biggr)|V(c_i,\xi)|^2,\\
\le C\frac{1}{\eps^{2q}\Vert \Delta(\cdot,\xi)\Vert_{L^\infty([a,b])}}\esp^{C\log(1/\eps)}|V(c_i,\xi)|^2,
\end{multline}
for all $t\in[c_i,d_i]$ and for $|\xi|\ge 1$, where we have used
Proposition \ref{propJT}, (ii), in the last step.
Taking into account that we have at most $p$ closed intervals $[c_i,d_i]$,
by combining \eqref{estEk} with \eqref{estEhyp2} we conclude that
\[
|V(b,\xi)|\le C\frac{1}{\eps^{pq}\Vert \Delta(\cdot,\xi)\Vert^{p/2}_{L^\infty([a,b])}}\esp^{C(\log(1/\eps)+\eps|\xi|)}|V(a,\xi)|,
\]
for $|\xi|\ge 1$. At this point setting $\eps=\esp^{-1}\lara{\xi}^{-1}$ we have that there exist constants $C'>0$ and $\kappa\in\N$ such that
\beq
\label{last-est}
|V(b,\xi)|\le C'\lara{\xi}^{pq+\kappa}|V(a,\xi)|,
\eeq
for $|\xi|\ge 1$. This proves the $\Cinf$ well-posedness of the Cauchy problem \eqref{CP}. Similarly, \eqref{last-est} implies the well-posedness of \eqref{CP} in $\mathcal{D}'(\R^n)$.
\end{proof}
\begin{remark}
Note that we can write the matrix $B$ as
\[
B(t,\xi)=\sum_{l=0}^{m-1}B_{-l}(t,\xi),
\]
where $B_{-l}$ has entries
\[
b_{-l,j}(t,\xi)=\begin{cases}
\sum_{|\gamma|=m-j-l}a_{m-j+1,\gamma}(t)\xi^\gamma\lara{\xi}^{j-m}, & \text{ for $j\le m-l$}, \\
0, & \text{otherwise},
\end{cases}
\]
of order $-l$ for $j=1,...,m$. It follows that
\begin{multline*}
\lara{(QB-B^\ast Q)V,V}=\sum_{i=1}^m \sum_{j=1}^m(q_{im}b_j-\overline{b_i}q_{jm})V_j\overline{V_i}\\
=\sum_{l=0}^{m-1}\sum_{i,j=1,...,m}(q_{im}b_{-l,j}-{\overline{b}}_{-l,i}q_{jm})V_j\overline{V_i}.
\end{multline*}
We notice that it is enough to put Levi conditions on terms up to order $-(m-2)$ for the 
$\Cinf$ well-posedness. More precisely, 
if $2q\le m-1$ it is enough to put Levi conditions on terms up to order $-(2q-1)$ in order
to get the $\Cinf$ well-posedness. Indeed, if
\[
|(q_{im}b_{-l,j}-{\overline{b}}_{-l,i}q_{jm})(t,\xi)|\le c\Delta(t,\xi),\qquad t\in[a,b],\, |\xi|\ge 1
\]
for $1\le i,j\le m$ and $0\le l\le 2q-1$ then by the hypothesis \eqref{GR1m} and Proposition \ref{propJT} we get the hyperbolic energy \[
\begin{split}
&\partial_t E(t,\xi)\\
&=\frac{\lara{\partial_t Q(t,\xi)V(t,\xi),V(t,\xi)}}{\lara{Q(t,\xi)V,V}}E(t,\xi)+i\lara{(Q(t,\xi)B(t,\xi)-B^\ast(t,\xi)Q(t,\xi))V(t,\xi),V(t,\xi)}\\
&\le C\sqrt{\frac{\wt{\Delta}(t,\xi)}{\Delta(t,\xi)}}E(t,\xi)+
\sum_{l=0}^{2q-1}|\lara{(Q(t,\xi)B_{-l}(t,\xi)-B_{-l}^\ast(t,\xi)Q(t,\xi))V(t,\xi),V(t,\xi)}|\\
&+\sum_{l=2q}^{m-1}|\lara{(Q(t,\xi)B_{-l}(t,\xi)-B_{-l}^\ast(t,\xi)Q(t,\xi))V(t,\xi),V(t,\xi)}|\\
&\le c\biggl(1+\frac{|\partial_t\Delta(t,\xi)|}{\Delta(t,\xi)}\biggr)E(t,\xi)+c\lara{\xi}^{-2q}\eps^{-2q}E(t,\xi).
\end{split}
\]
Note that
\[
\sum_{l=0}^{2q-1}|\lara{(Q(t,\xi)B_{-l}(t,\xi)-B_{-l}^\ast(t,\xi)Q(t,\xi))V(t,\xi),V(t,\xi)}|
\]
is estimated by means of the Levi conditions, whereas the bound $\lara{\xi}^{-2q}\eps^{-2q}E(t,\xi)$ for
\[
\sum_{l=2q}^{m-1}|\lara{(Q(t,\xi)B_{-l}(t,\xi)-B_{-l}^\ast(t,\xi)Q(t,\xi))V(t,\xi),V(t,\xi)}|
\]
is obtained by symbol properties. At this point setting $\eps=\esp^{-1}\lara{\xi}^{-1}$ we can estimate $|V(t,\xi)|^2$ as in the proof of Theorem \ref{theoGRm}.
\end{remark}


\section{Examples and remarks on the condition \eqref{GR1m} and the Levi conditions \eqref{EQ:Leco}}
\label{S:4}

In this section we collect some examples and we have a closer look at the hypothesis
\eqref{GR1m} and the Levi conditions \eqref{LCB} (or \eqref{EQ:Leco}).

We begin with the hypothesis \eqref{GR1m}. For the sake of simplicity we will
assume that the roots have only one zero at $t=0$ (the proof can be adapted to the case
of a finite number of zeroes) and we will take the interval $[a,b]=[0,T]$.

Recall that in this case there exists positive constants $c_1,c_2$ such that
\beq
\label{Delta2}
c_1\frac{1}{t^2}\Delta(t,\xi)\le \wt{\Delta}(t,\xi)\le c_2\frac{1}{t^2}\Delta(t,\xi),
\eeq
for $t\in[0,T]$ and $|\xi|\ge 1$. 
Before we argue this, we note that we will write such bounds over intervals of the type $[0,T]$, meaning that
they hold over $(0,T]$ and extend uniformly over $t=0$ (usually due to cancellation of zeros).
Now, the estimate \eqref{Delta2} is trivial when $t$ is far from $0$, i.e. $t\in[\beta, T]\subseteq [0,T]$ with $\beta>0$, since both $\Delta(t,\xi)$ and $\wt{\Delta}(t,\xi)$ are different from 0 there. When we are on a sufficiently small interval $[0,\beta]$ using the analyticity in $t$ of $\Delta(t,\xi)$ we have
\beq
\label{Delta1}
t^2\frac{(\partial_t\Delta(t,\xi))^2}{\Delta^2(t,\xi)}\in L^\infty([0,\beta], |\xi|\ge 1)
\eeq
and
\beq
\label{Delta-two}
t^2\frac{\Delta^2(t,\xi)}{(\partial_t\Delta(t,\xi))^2}\in L^\infty([0,\beta], |\xi|\ge 1).
\eeq
To prove \eqref{Delta1} recall the relation \eqref{detq} and write $\Delta(t,\xi)$ as $t^k g(t,\xi)$, where $k=k(\xi)$ is positive and bounded and $|g(t,\xi)|\ge \gamma_0$ for all $t\in[0,T]$ and $|\xi|\ge 1$. Hence,
\[
t^2\frac{(\partial_t\Delta(t,\xi))^2}{\Delta^2(t,\xi)}= \frac{t^2 t^{2k-2}(kg(t,\xi)+t\partial_t g(t,\xi))^2}{t^{2k}g^2(t,\xi)}\le c
\]
for $t\in[0,\beta]$ and $|\xi|\ge 1$. Analogously, to prove \eqref{Delta-two} setting $\sup_{t\in[0,T], |\xi|\ge 1}|g(t,\xi)|=\gamma_1$ and $\sup_{t\in[0,T], |\xi|\ge 1}|\partial_tg(t,\xi)|=\gamma_2$ we get
\begin{multline*}
t^2\frac{\Delta^2(t,\xi)}{(\partial_t\Delta(t,\xi))^2}=\frac{t^2t^{2k}g^2(t,\xi)}{t^{2k-2}(kg(t,\xi)+t\partial_t g(t,\xi))^2}\\
\le \frac{t^{2k+2}}{t^{2k-2}} \frac{g^2(t,\xi)}{(k|g(t,\xi)|-t|\partial_t g(t,\xi)|)^2}\le t^4 \frac{\gamma_1^2}{(k\gamma_0-t\gamma_2)^2},
\end{multline*}
which is bounded on $[0,\beta]\times\{\xi: |\xi|\ge 1\}$ choosing $\beta$ small enough. Thus \eqref{Delta2} is valid for $t\in[0,\beta]$ and $|\xi|\ge 1$.

We are now ready to state the following proposition.
\begin{proposition}
\label{prop_psi}
Let $m=2$ and let
$\lambda_1(t,\xi)=\tau_1(t,\xi)\lara{\xi}^{-1}, \lambda_2(t,\xi)=\tau_2(t,\xi)\lara{\xi}^{-1}$
be the renormalised roots of the characteristic polynomial of \eqref{CP}.
Assume that $\lambda_1(t,\xi)$ and $\lambda_2(t,\xi)$ which are analytic in
$t$ coincide at $t=0$ only, with $\lambda_1(0,\xi)=\lambda_2(0,\xi)=0$.
Then, the hypothesis \eqref{GR1m} is equivalent to each of the following two
conditions, for $|\xi|\ge 1$ and $T$ small enough:
\begin{itemize}
\item[(i)] there exists $M>0$ such that
\[
t^2(\partial_t\lambda_1(t,\xi))^2+t^2(\partial_t\lambda_2(t,\xi))^2\le M(\lambda_1(t,\xi)-\lambda_2(t,\xi))^2,\qquad t\in[0,T],\, |\xi|\ge 1;
\]
\item[(ii)] there exists $M>0$ such that
\[
\lambda_1^2(t,\xi)+\lambda_2^2(t,\xi)\le M(\lambda_1(t,\xi)-\lambda_2(t,\xi))^2,\qquad t\in[0,T],\, |\xi|\ge 1.
\]
\end{itemize}
\end{proposition}
\begin{proof}
We recall that $\tau_1(t,\xi)$ and $\tau_2(t,\xi)$ are the roots of the characteristic polynomial
\[
\tau^2-\sum_{|\nu|=1} a_{\nu,|\nu|}(t)\xi^\nu\tau-\sum_{|\nu|=2} a_{\nu,|\nu|}(t)\xi^\nu.
\]
Hence,
\[
\lambda_1(t,\xi)=\frac{\sum_{|\nu|=1} a_{\nu,|\nu|}(t)\xi^\nu\lara{\xi}^{-1}-\sqrt{\Delta(t,\xi)}}{2}
\]
and
\[
\lambda_2(t,\xi)=\frac{\sum_{|\nu|=1} a_{\nu,|\nu|}(t)\xi^\nu\lara{\xi}^{-1}+\sqrt{\Delta(t,\xi)}}{2},
\]
where
\[
\Delta(t,\xi)=\biggl(\sum_{|\nu|=1} a_{\nu,|\nu|}(t)\xi^\nu\biggr)^2\lara{\xi}^{-2}+4\sum_{|\nu|=2} a_{\nu,|\nu|}(t)\xi^\nu\lara{\xi}^{-2}.
\]
Clearly,
\[
(\lambda_1(t,\xi)-\lambda_2(t,\xi))^2=\Delta(t,\xi).
\]
Moreover,
\[
(\partial_t\lambda_1(t,\xi))^2+(\partial_t\lambda_2(t,\xi))^2 \le c (\partial_t\lambda_1(t,\xi)+\partial_t\lambda_2(t,\xi))^2+c(\partial_t\sqrt{\Delta(t,\xi)})^2.
\]
By definition $|\psi(t,\xi)|=(\partial_t\lambda_1(t,\xi)+\partial_t\lambda_2(t,\xi))^2$, hence \eqref{GR1m} implies
\beq
\label{estimate2}
(\partial_t\lambda_1(t,\xi))^2+(\partial_t\lambda_2(t,\xi))^2 \le c \biggl(\Delta(t,\xi)+\frac{(\partial_t\Delta(t,\xi))^2}{\Delta(t,\xi)}\biggl)=c\wt{\Delta}(t,\xi).
\eeq
From the previous observation \eqref{Delta2} on $\wt{\Delta}$ we have that $\wt{\Delta}(t,\xi)\le c_2(1/t^2)\Delta(t,\xi)$ on $[0,T]$ for $|\xi|\ge 1$. Applied to \eqref{estimate2} this proves that \eqref{GR1m} implies $(i)$.

Assume that $(i)$ holds. By Taylor's expansion and adding and subtracting the term $\partial_t\lambda_1(\theta_2t,\xi)$ we can write
\begin{multline*}
(\lambda_1(t,\xi))^2+(\lambda_2(t,\xi))^2=
(\partial_t\lambda_1(\theta_1t,\xi)t)^2+(\partial_t\lambda_2(\theta_2t,\xi)t)^2\\
\le 2t^2(\partial_t\lambda_1(\theta_1t,\xi)-\partial_t\lambda_1(\theta_2t,\xi))^2+
2t^2(\partial_t\lambda_1(\theta_2t,\xi))^2+t^2(\partial_t\lambda_2(\theta_2t,\xi))^2.
\end{multline*}
The hypothesis $(i)$ yields
\begin{multline}
\label{est1-2}
(\lambda_1(t,\xi))^2+(\lambda_2(t,\xi))^2 \le \\
c t^2(\partial_t\lambda_1(\theta_1t,\xi)-\partial_t\lambda_1(\theta_2t,\xi))^2 +
c (\lambda_1(\theta_2t,\xi)-\lambda_2(\theta_2t,\xi))^2.
\end{multline}
Both the summands in the right hand-side of \eqref{est1-2} are bounded by $(\lambda_1(t,\xi)-\lambda_2(t,\xi))$. Indeed, recalling that $\lambda_1(t,\xi)=t^{k_1}g_1(t,\xi)$, $\lambda_2(t,\xi)=t^{k_2}g_2(t,\xi)$, where $g_1(t,\xi)\neq 0$, $g_2(t,\xi)\neq 0$ and the functions $k_1(\xi), k_2(\xi)$ are positive and bounded, we obtain
\begin{multline*}
t^2\frac{(\partial_t\lambda_1(\theta_1t,\xi)-\partial_t\lambda_1(\theta_2t,\xi))^2}{\Delta(t,\xi)}\\
= t^2\frac{(k_1(\theta_1t)^{k_1-1}g_1(\theta_1t,\xi)+(\theta_1 t)^{k_1}\partial_t g_1(\theta_1 t,\xi)-k_2(\theta_2t)^{k_2-1}g_2(\theta_2t,\xi)-(\theta_2 t)^{k_2}\partial_t g_2(\theta_2 t,\xi))^2}{(t^{k_1}g_1(t,\xi)-t^{k_2}g_2(t,\xi))^2}\\
\le \frac{c_{g_1} t^{2k_1}+ c_{g_2}t^{2k_2}}{t^{2k}(t^{k_1-k}g_1(t,\xi)-t^{k_2-k}g_2(t,\xi))^2},
\end{multline*}
where $k=\min\{k_1,k_2\}$. Since the functions $g_1$ and $g_2$ are bounded with $|g_1(t,\xi)|\ge c_1$ and $|g_2(t,\xi)|\ge c_2$ for $t\in[0,T]$ and $|\xi|\ge 1$, we have that the bound from below
\[
|t^{k_1-k}g_1(t,\xi)-t^{k_2-k}g_2(t,\xi)|\ge c_3
\]
holds uniformly in $|\xi|\ge 1$ and $t\in[0,T]$, when $T$ is sufficiently small. Thus,
\[
t^2\frac{(\partial_t\lambda_1(\theta_1t,\xi)-\partial_t\lambda_1(\theta_2t,\xi))^2}{\Delta(t,\xi)}\le \frac{c_{g_1} t^{2k_1}+ c_{g_2}t^{2k_2}}{t^{2k}(t^{k_1-k}g_1(t,\xi)-t^{k_2-k}g_2(t,\xi))^2}\le C.
\]
Analogously,
\begin{multline*}
\frac{(\lambda_1(\theta_2t,\xi)-\lambda_2(\theta_2t,\xi))^2}{\Delta(t,\xi)}=
\frac{(\lambda_1(\theta_2t,\xi)-\lambda_2(\theta_2t,\xi))^2}{(\lambda_1(t,\xi)-\lambda_2(t,\xi))^2}\\
= \frac{((\theta_2t)^{k_1}g_1(\theta_2t,\xi)-(\theta_2t)^{k_2}g_2(\theta_2t,\xi))^2}
{(t^{k_1}g_1(t,\xi)-t^{k_2}g_2(t,\xi))^2}\\
\le \frac{c_{g_1}t^{2k_1}+c_{g_2}t^{2k_2}}{t^{2k}(t^{k_1-k}g_1(t,\xi)-t^{k_2-k}g_2(t,\xi))^2}\le C.
\end{multline*}
This completes the proof of $(i)$ implies $(ii)$.

It remains to prove that $(ii)$ implies \eqref{GR1m}. We recall that from Lemma 6.2 in \cite{GR:12}, applied to  $\partial_t\lambda_1$ and $\partial_t\lambda_2$, the estimate
\[
t^2|\partial_t\lambda_i(t,\xi)|\le c |\lambda_i(t,\xi)|
\]
holds for $i=1,2$, uniformly in $t\in[0,T]$ and $|\xi|\ge 1$. Then, under the hypothesis $(ii)$ we get
\begin{multline*}
t^2|\psi(t,\xi)|=t^2(\partial_t\lambda_1(t,\xi)+\lambda_2(t,\xi))^2\le 2t^2(\partial_t\lambda_1(t,\xi))^2+2t^2(\partial_t\lambda_2(t,\xi))^2
\\
\le c\lambda_1^2(t,\xi)+c\lambda_2^2(t,\xi)\le M(\lambda_1(t,\xi)-\lambda_2(t,\xi))^2 =M \Delta(t,\xi).
\end{multline*}
Finally, from \eqref{Delta2} we conclude that $|\psi(t,\xi)|\le C\wt{\Delta}(t,\xi)$. This completes the proof.
\end{proof}

We assume now for simplicity that the lower order terms in \eqref{CP} are real-valued.
We want to compare the Levi condition in \eqref{LCbr} with the corresponding Levi condition
introduced in \cite{GR:12}, which assumes
\beq
\label{LC2}
|b_j(t,\xi)|^2\le c q_{jj}(t,\xi),\qquad j=1,...,m.
\eeq
In general \eqref{LC2} is weaker then \eqref{LCbr}, in the sense that \eqref{LCbr} implies \eqref{LC2}. 
However, for the second order equations we have:
\begin{proposition}
\label{propLC5}
Let $m=2$. If the roots $\lambda_1, \lambda_2$ fulfil
\beq
\label{cond2}
\lambda_1^2(t,\xi)+\lambda_2^2(t,\xi)\le M (\lambda_1(t,\xi)-\lambda_2(t,\xi))^2,\qquad t\in[a,b],\, |\xi|\ge 1,
\eeq
 then the Levi condition \eqref{LC2} is equivalent to
\beq
\label{LCB2}
|(q_{12}b_2-b_1q_{22})(t,\xi)|^2\le c\,\Delta(t,\xi)
\eeq
for $t\in[a,b]$ and $|\xi|\ge 1$.
\end{proposition}
\begin{proof}
We recall that in this case the symmetriser is given by the matrix
\[
Q=\left(
    \begin{array}{cc}
      \lambda_1^2+\lambda_2^2 & -\lambda_1-\lambda_2\\
       -\lambda_1-\lambda_2 & 2\\
    \end{array}
  \right).
\]
From \eqref{cond2} it is clear that $\Delta(t,\xi)$ is equivalent to $\lambda_1^2(t,\xi)+\lambda_2^2(t,\xi)$. We begin by writing
\begin{multline*}
|2b_1(t,\xi)|^2\le |q_{22}(t,\xi)b_1(t,\xi)-q_{12}(t,\xi)b_2(t,\xi)+q_{12}(t,\xi)b_2(t,\xi)|^2\\
\le 2|q_{22}(t,\xi)b_1(t,\xi)-q_{12}(t,\xi)b_2(t,\xi)|^2+2|q_{12}(t,\xi)b_2(t,\xi)|^2.
\end{multline*}
Using the definition of the entries of $Q$ given above, the fact that $\Delta$ is equivalent to $\lambda_1^2+\lambda_2^2$ and the boundedness of $b_2$, we obtain under the assumption \eqref{LCB2} that
\[
|b_1(t,\xi)|^2\le c (\Delta(t,\xi)+(\lambda_1(t,\xi)+\lambda_2(t,\xi))^2)\le c'(\lambda_1^2(t,\xi)+\lambda_2^2(t,\xi)) = c'\, q_{11}(t,\xi).
\]
This shows that \eqref{LCB2} implies \eqref{LC2} for $j=1$. For $j=2$ condition \eqref{LC2} means $|b_2(t,\xi)|\le C$ which is trivially satisfied. Conversely, if \eqref{LC2} holds then
\begin{multline*}
|q_{22}(t,\xi)b_1(t,\xi)-q_{12}(t,\xi)b_2(t,\xi)|^2\le 2 |q_{22}(t,\xi)b_1(t,\xi)|^2+2|q_{12}(t,\xi)b_2(t,\xi)|^2\\
\le c (q_{11}(t,\xi)+q_{12}^2(t,\xi))\le c' (\lambda_1^2(t,\xi)+\lambda_2^2(t,\xi)),
\end{multline*}
which implies that \eqref{LCB2} is valid.
\end{proof}
Clearly $\Delta^2(t,\xi)\le c\Delta(t,\xi)$ when $t\in[a,b]$ and $|\xi|\ge 1$, so \eqref{LCbr}$\Rightarrow$\eqref{LCB2}$\Leftrightarrow$\eqref{LC2} under the assumption on the roots \eqref{cond2}.
Proposition \ref{propLC5} also shows that the Levi conditions on the lower order terms can be improved by making use of the hypothesis \eqref{GR1m}. This requires a precise knowledge of the relation between the roots of the equation and the check function $\psi$.

\bibliographystyle{abbrv}
\newcommand{\SortNoop}[1]{}

\end{document}